\documentclass[12pt,reqno]{amsart}
\usepackage{amssymb}
\usepackage{amsxtra}
\usepackage{mathrsfs}
\usepackage[all]{xy}
\usepackage{cite}
\usepackage{paralist}
\usepackage[hypertex]{hyperref}
\numberwithin{equation}{section}
\newtheorem{theorem}{Theorem}[section]
\newtheorem{lemma}[theorem]{Lemma}
\newtheorem{prop}[theorem]{Proposition}
\newtheorem{corollary}[theorem]{Corollary}
\newtheorem{problem}[theorem]{Problem}
\theoremstyle{definition}
\newtheorem{definition}[theorem]{Definition}
\theoremstyle{remark}
\newtheorem{example}[theorem]{Example}
\newtheorem{remark}[theorem]{Remark}

\DeclareMathOperator{\Hom}{Hom}

\DeclareMathOperator{\Pol}{Pol}
\DeclareMathOperator{\Fun}{Fun}

\newcommand*{\op}{\mathrm{op}}
\newcommand*{\reg}{\mathrm{reg}}

\newcommand*{\wh}{\widehat}

\newcommand*{\ra}{\rangle}

\newcommand*{\CC}{\mathbb C}
\newcommand*{\N}{\mathbb N}
\newcommand*{\Z}{\mathbb Z}
\newcommand*{\R}{\mathbb R}
\newcommand*{\DD}{\mathbb D}
\newcommand*{\BB}{\mathbb B}

\newcommand*{\MM}{\mathbb M}
\newcommand*{\cO}{\mathscr O}

\newcommand*{\cF}{\mathscr F}

\newcommand*{\rp}{\mathrm p}
\newcommand*{\rt}{\mathrm t}

\newcommand*{\rV}{\mathrm V}

\newcommand*{\spn}{\mathrm{span}}

\newcommand*{\Alg}{\mathsf{Alg}}
\newcommand*{\AM}{\mathsf{AM}}

\newcommand*{\ol}{\overline}

{\end{compactenum}}
\newcommand*{\xra}{\xrightarrow}
\newcommand{\lriso}{\stackrel{\textstyle\sim}{\smash\longrightarrow%
\vphantom{\scriptscriptstyle{_1}}}}

\begin{document}
\title[Quantum polydisk and quantum ball]{Holomorphic functions on the quantum polydisk\\
and on the quantum ball}
\author{A. Yu. Pirkovskii}
\address{Faculty of Mathematics\\
National Research University Higher School of Economics\\
6 Usacheva, 119048 Moscow, Russia}
\email{aupirkovskii@hse.ru, pirkosha@gmail.com}
\date{}

\begin{abstract}
We introduce and study noncommutative (or ``quantized'') versions of the algebras
of holomorphic functions on the polydisk and on the ball in $\CC^n$.
Specifically, for each $q\in\CC\setminus\{ 0\}$ we construct Fr\'echet algebras
$\cO_q(\DD^n)$ and $\cO_q(\BB^n)$ such that for $q=1$ they are isomorphic to the
algebras of holomorphic functions on the open polydisk $\DD^n$ and on the open ball $\BB^n$,
respectively. In the case where $0<q<1$, we establish a relation between our
holomorphic quantum ball algebra $\cO_q(\BB^n)$ and L.\,L.\,Vaksman's algebra
$C_q(\bar\BB^n)$ of continuous functions on the closed quantum ball.
Finally, we show that $\cO_q(\DD^n)$ and $\cO_q(\BB^n)$ are not isomorphic
provided that $|q|=1$ and $n\ge 2$. This result can be interpreted as a $q$-analog of
Poincar\'e's theorem, which asserts that $\DD^n$ and $\BB^n$ are not biholomorphically equivalent
unless $n=1$.
\end{abstract}

\maketitle

\section{Introduction}
\label{sect:intro}

The subject of the present paper may be roughly described as ``noncommutative complex analysis'',
or ``noncommutative complex analytic geometry''.
This field of mathematics is not as unified as other parts of noncommutative
geometry, and there are many points of view on what noncommutative
complex analysis is. At present, there are several ``branches'' of noncommutative complex analysis,
and they are largely independent of each other. Here is a brief overview of some of them.

\smallskip
\begin{asparaenum}[$\bullet$]
\item
An operator algebra approach to noncommutative complex analysis, which
goes back to the foundational
papers \cite{Arv_anal,Arv_subalg_I} of W.\,B.\,Arveson.
Important contributions to this field include a series of recent publications
by G.\,Popescu (see, e.g., \cite{Pop_disk,Pop_holball,Pop_var,Pop_ncdom,Pop_aut_poly}, to cite a few),
by P.\,Muhly and B.\,Solel
(see, e.g., \cite{MS_tens_corr,MS_Hardy,MS_Schur,MS_tensfunc}).
See also \cite{Dav_FSG,Dav_Pop,Dav_bihol,Dav_isoprobl,Ar_Lat_iso,Ar_Lat_iso2}.
\item
The theory of free holomorphic functions initiated by J.\,L.\,Taylor \cite{T2,T3}
and developed in \cite{Lum_PI,Lum_matr,Voic1,Voic2,AMC_glob,HCMCS,HCMC_pencil}, for example.
The current state of this subject is thoroughly presented in a recent monograph
\cite{KVV_book} by D.\,S.\,Kaliuzhnyi-Verbovetskyi and V.\,Vinnikov.
\item
An algebraic view of noncommutative complex geometry
based on A.~Connes' fundamental ideas \cite{Connes_Cst,Connes_NCG} of
noncommutative connections and positive Hochschild cocycles.
Among numerous publications in this field, we mention
\cite{DS,PolShcw_cats,Pol_class,Pol_qsicoh,Rosenbg_Lapl,KLvS,KM1,KM2,
Beggs_Smith,Buach_qbnd,Buach_hmg}.
A detailed discussion of this side of noncommutative complex geometry is given
in \cite{KLvS} and \cite{Beggs_Smith}.
\item
The theory of quantum bounded symmetric domains founded by L.\,L.\,Vaksman and his
collaborators. The pioneering paper in this field is \cite{Vaks_Dolb}, and complex-analytic
aspects of this theory are emphasized in \cite{Vaks_max,Prosk_Tur_shil,Joh_Tur,Ber_Gis_Tur},
for example. For more references and
further information on quantum bounded symmetric domains, see the
lecture notes \cite{Vaks_sborn} and Vaksman's recent monograph \cite{Vaks_book}.
\item
A study of noncommutative Cauchy-Riemann equations, both in the functional-analytic
\cite{RW,Szafr} and in the algebraic \cite{BDR} contexts.
\item
The theory of DQ-modules initiated by M.\,Kashiwara and P.\,Schapira \cite{KS_dqmod}.
\end{asparaenum}

\smallskip

Needless to say, noncommutative complex analysis, as well as the whole field
of noncommutative geometry, was strongly influenced by the foundational
papers of V.\,G.\,Drinfeld \cite{Drinf} and Yu.\,I.\,Manin \cite{Manin_Kosz,Manin_NCG}.

A more detailed comparison of various branches of noncommutative complex analysis
can be found in \cite{Pir_qball_arx}.

Our approach to noncommutative complex analysis is slightly different
from what was mentioned above. Broadly speaking,
the objects we are mostly interested in are nonformal deformations
of the algebras of holomorphic functions on complex Stein manifolds.
The present paper is only the first step in studying such objects.
Specifically, we concentrate on deformations of the algebras of holomorphic
functions on the open polydisk and the open ball
in $\CC^n$. We hope that these concrete examples can serve as a basis
for further research in noncommutative complex analysis.
To motivate our constructions, let $q\in\CC\setminus\{ 0\}$, and consider the
algebra $\cO_q^\reg(\CC^n)$ of ``regular functions on the quantum affine space'' generated by
$n$ elements $x_1,\ldots ,x_n$ subject to the relations $x_i x_j=qx_j x_i$ for all $i<j$
(see, e.g., \cite{Br_Good}). If $q=1$, then $\cO_q^\reg(\CC^n)$ is nothing but the
polynomial algebra $\CC[x_1,\ldots ,x_n]=\cO^\reg(\CC^n)$.
It is a standard fact that the monomials $x_1^{k_1}\cdots x_n^{k_n}$
$(k_1,\ldots ,k_n\ge 0)$ form a basis of $\cO_q^\reg(\CC^n)$,
so the underlying vector space of $\cO_q^\reg(\CC^n)$ can be identified with that
of $\cO^\reg(\CC^n)$.

Let now $r\in (0,+\infty]$, and let $\DD^n_r$ and $\BB^n_r$ denote the open polydisk
and the open ball of radius $r$ in $\CC^n$. Thus we have
\begin{align*}
\DD^n_r&=\Bigl\{ z=(z_1,\ldots ,z_n)\in\CC^n : \max_{1\le i\le n} |z_i|<r\Bigr\},\\
\BB^n_r&=\Bigl\{ z=(z_1,\ldots ,z_n)\in\CC^n : \sum_{i=1}^n |z_i|^2<r^2\Bigr\}.
\end{align*}
Since $\cO^\reg(\CC^n)$ is dense both in $\cO(\DD^n_r)$ and in $\cO(\BB^n_r)$, it seems
reasonable to define the algebras of holomorphic functions on the quantum polydisk and on the quantum ball
as certain completions of $\cO_q^\reg(\CC^n)$. Intuitively, the idea is to ``deform''
the pointwise multiplication on $\cO(\DD^n_r)$ and $\cO(\BB^n_r)$ in such a way
that $z_i z_j=qz_j z_i$ for all $i<j$. This idea is too naive, however, because it is not immediate whether
there exists a multiplication satisfying the above condition.
In fact, to give a ``correct'' definition
of our quantized algebras, we have to ``deform'' not only the multiplication,
but also the underlying topological vector spaces of $\cO(\DD^n_r)$ and $\cO(\BB^n_r)$
(see Remark~\ref{rem:no_mult}).

The structure of the paper is as follows. After giving some preliminaries in Section~\ref{sect:prelim},
we proceed in Section~\ref{sect:qball} to define our quantized function algebras
$\cO_q(\DD^n_r)$ and $\cO_q(\BB^n_r)$. The algebra $\cO_q(\DD^n_r)$ was introduced
earlier in \cite{Pir_HFG} in a more general multiparameter case, while the (more involved)
definition of $\cO_q(\BB^n_r)$ is new. In the same section we show that the algebra
isomorphism $\cO_q^\reg(\CC^n)\to\cO_{q^{-1}}^\reg(\CC^n)$ given by $x_i\mapsto x_{n-i}$
extends to $\cO_q(\DD^n_r)$ and $\cO_q(\BB^n_r)$.
In Section~\ref{sect:qball_Vaks}, we compare $\cO_q(\BB^n_r)$ (in the special case where $0<q<1$)
with L.\,L.\,Vaksman's algebra of continuous functions on the closed quantum ball \cite{Vaks_max}.
Roughly, our result is that $\cO_q(\BB^n_r)$ (similarly to its ``classical'' prototype)
is isomorphic to the completion of $\cO_q^\reg(\CC^n)$
with respect to the ``quantum $\sup$-norms'' over the closed balls contained in $\BB^n_r$.
In Section~\ref{sect:Poincare}, we show that $\cO_q(\DD^n_r)$ and $\cO_q(\BB^n_r)$ are
topologically isomorphic if $|q|\ne 1$, but are not topologically isomorphic if $|q|=1$,
$n\ge 2$ and $r<\infty$. The latter result may be viewed as a $q$-analog of Poincar\'e's theorem,
which asserts that $\DD^n_r$ and $\BB^n_r$ are not biholomorphically equivalent.

In future publications, we plan to develop a deformation-theoretic
approach to our quantum polydisk and quantum ball algebras and to show, in particular, that they
form strict Fr\'echet deformation quantizations of $\cO(\DD^n_r)$ and
$\cO(\BB^n_r)$ in the sense of M.\,A.\,Rieffel (see, e.g., \cite{Rf_mem}).
We believe that such interpretations of $\cO_q(\DD^n_r)$ and $\cO_q(\BB^n_r)$
will indicate that our {\em ad hoc} definitions given in the present paper
are indeed the ``correct'' ones.

This paper is a part of the author's preprint \cite{Pir_qball_arx}.
Some of the results presented here were announced in \cite{Pir_ISQS21}.

\section{Preliminaries and notation}
\label{sect:prelim}

We shall work over the field $\CC$ of complex numbers. All algebras are assumed to
be associative and unital, and all algebra homomorphisms are assumed to be unital
(i.e., to preserve identity elements).
By a {\em Fr\'echet algebra} we mean a complete metrizable locally convex
algebra (i.e., a topological algebra whose underlying space
is a Fr\'echet space). A {\em locally $m$-convex algebra} \cite{Michael} is a topological
algebra $A$ whose topology can be defined by a family of submultiplicative
seminorms (i.e., seminorms $\|\cdot\|$ satisfying $\| ab\|\le \| a\| \| b\|$
for all $a,b\in A$). A complete locally $m$-convex algebra is called
an {\em Arens-Michael algebra} \cite{X2}.

Throughout we will use the following multi-index notation. Let $\Z_+=\N\cup\{ 0\}$ denote the set
of all nonnegative integers. For each $n\in\N$ and each $d\in \Z_+$, let $W_{n,d}=\{ 1,\ldots ,n\}^d$,
and let $W_n=\bigsqcup_{d\in\Z_+} W_{n,d}$. Thus a typical element of $W_n$
is a $d$-tuple $\alpha=(\alpha_1,\ldots ,\alpha_d)$ of arbitrary length
$d\in\Z_+$, where $\alpha_j\in\{ 1,\ldots ,n\}$ for all $j$. The only element
of $W_{n,0}$ will be denoted by $*$.
For each $\alpha\in W_{n,d}\subset W_n$, let $|\alpha|=d$.
Given an algebra $A$, an $n$-tuple $a=(a_1,\ldots ,a_n)\in A^n$, and
$\alpha=(\alpha_1,\ldots ,\alpha_d)\in W_n$, we let
$a_\alpha=a_{\alpha_1}\cdots a_{\alpha_d}\in A$ if $d>0$;
it is also convenient to set $a_*=1\in A$.
Given $k=(k_1,\ldots ,k_n)\in\Z_+^n$, we let $a^k=a_1^{k_1}\cdots a_n^{k_n}$.
As usual, for each $k=(k_1,\ldots ,k_n)\in\Z_+^n$, we let $|k|=k_1+\cdots +k_n$.

We will also use the standard notation related to $q$-numbers
(see, e.g., \cite{Kac_Ch,Kassel,Gasp_Rahm}).
Given $q\in\CC^\times=\CC\setminus\{ 0\}$ and $k\in\N$, let
\[
[k]_q=1+q+\cdots +q^{k-1};\quad [k]_q!=[1]_q [2]_q \cdots [k]_q.
\]
It is also convenient to let $[0]_q!=1$. If $k=(k_1,\ldots ,k_n)\in\Z_+^n$, then
we let $[k]_q!=[k_1]_q!\cdots [k_n]_q!$. If $|q|<1$, then for each $a\in\CC$
we let $(a;q)_\infty=\prod_{j=0}^\infty (1-aq^j)$.

Given an algebra $A$, the {\em Arens-Michael envelope}
of $A$ is the completion of
$A$ with respect to the family of all submultiplicative seminorms on $A$.
Equivalently, the Arens-Michael envelope of $A$ is an Arens-Michael algebra $\wh{A}$
together with a natural isomorphism
\begin{equation}
\label{AM_def}
\Hom_\Alg(A,B)\cong \Hom_\AM(\wh{A},B)\qquad (B\in\AM),
\end{equation}
where $\Alg$ is the category of algebras and algebra homomorphisms,
and $\AM$ is the category of Arens-Michael algebras and continuous algebra
homomorphisms.
Moreover, the correspondence $A\mapsto\wh{A}$ is a functor from $\Alg$
to $\AM$, and this functor is left adjoint to the forgetful functor $\AM\to\Alg$
(see~\eqref{AM_def}).

Arens-Michael envelopes were introduced by J.\,L.\,Taylor \cite{T1} under the name of
``completed locally $m$-convex envelopes''. Now it is customary to call them
``Arens-Michael envelopes'', following the terminology suggested by A.\,Ya.\,Helemskii \cite{X2}.

Here are some basic examples of Arens-Michael envelopes.
If $A=\CC[z_1,\ldots ,z_n]$ is the polynomial
algebra, then the Arens-Michael envelope of $A$ is the algebra $\cO(\CC^n)$
of entire functions on $\CC^n$ \cite{T2}. More generally \cite[Example 3.6]{Pir_qfree},
if $X$ is an affine algebraic variety and $A=\cO^\reg(X)$ is the algebra of regular functions on $X$,
then $\wh{A}$ is the algebra $\cO(X)$ of holomorphic functions on $X$.
A similar result holds in the case where $(X,\cO^\reg_X)$ is an affine scheme
of finite type over $\CC$ \cite[Example 7.2]{Pir_HFG}.

Let $q\in\CC^\times$.
Recall from Section~\ref{sect:intro} that the algebra $\cO_q^\reg(\CC^n)$
{\em of regular functions on the quantum affine $n$-space} is generated by
$n$ elements $x_1,\ldots ,x_n$ subject to the relations $x_i x_j=qx_j x_i$ for all $i<j$
(see, e.g., \cite{Br_Good}).
Using the above examples as a motivation, we define \cite{Pir_qfree} the algebra
$\cO_q(\CC^n)$ {\em of holomorphic functions
on the quantum affine space} to be the
Arens-Michael envelope of $\cO_q^\reg(\CC^n)$.
As was shown in \cite[Theorem 5.11]{Pir_qfree}, we have
\begin{equation}
\label{O_q}
\cO_q(\CC^n)=
\Bigl\{
a=\sum_{k\in\Z_+^n} c_k x^k :
\| a\|_\rho=\sum_{k\in\Z_+^n} |c_k| w_q(k) \rho^{|k|}<\infty
\;\forall \rho>0
\Bigr\},
\end{equation}
where
\begin{equation}
\label{w_q}
w_q(k)=
\begin{cases}
1 & \text{if } |q|\ge 1;\\
|q|^{\sum_{i<j}k_i k_j} & \text{if } |q|<1.
\end{cases}
\end{equation}
The topology on $\cO_q(\CC^n)$ is given by the norms $\|\cdot\|_\rho \; (\rho>0)$,
and the multiplication is uniquely determined by $x_i x_j=qx_j x_i$ ($i<j$).
Moreover, each norm $\|\cdot\|_\rho$ is submultiplicative.

We refer to \cite{T2,Pir_qfree,Dos_hdim,Dos_absbas}
for explicit descriptions of Arens-Michael envelopes
of some other finitely generated algebras, including quantum tori, quantum Weyl algebras,
the algebra of quantum $2\times 2$-matrices, and universal enveloping algebras.
Further results on Arens-Michael envelopes can be found in
\cite{Pir_stbflat,Pir_locsolv,Dos_holfun,Dos_Slodk,Dos_locleft}.

\section{Quantum polydisk and quantum ball}
\label{sect:qball}

Let us start by recalling a well-known power series characterization of the algebra
$\cO(\DD^n_r)$ of holomorphic functions on the polydisk (see, e.g., \cite[Example 27.27]{MV}). We have
\begin{equation}
\label{poly_power_rep}
\cO(\DD^n_r)\cong
\Bigl\{
a=\sum_{k\in\Z_+^n} c_k z^k :
\| a\|_{\DD,\rho}=\sum_{k\in\Z_+^n} |c_k| \rho^{|k|}<\infty
\;\forall \rho\in (0,r) \Bigr\}.
\end{equation}
The space on the right-hand side of \eqref{poly_power_rep} is a subalgebra
of $\CC[[z_1,\ldots ,z_n]]$ and a Fr\'echet-Arens-Michael algebra under the family
$\{ \|\cdot\|_{\DD,\rho} : \rho\in (0,r)\}$ of submultiplicative norms.
The Fr\'echet algebra isomorphism~\eqref{poly_power_rep} takes each holomorphic
function on $\DD^n_r$ to its Taylor expansion at $0$.

The following definition is motivated by \eqref{O_q} and \eqref{poly_power_rep}.

\begin{definition}[\cite{Pir_ncStein,Pir_HFG}]
\label{def:q_poly}
Let $q\in\CC^\times$, and let $r\in (0,+\infty]$.
The {\em algebra of holomorphic functions on the
quantum $n$-polydisk of radius $r$} is
\begin{equation*}
\cO_q(\DD^n_r)=
\Bigl\{
a=\sum_{k\in\Z_+^n} c_k x^k :
\| a\|_{\DD,\rho}=\sum_{k\in\Z_+^n} |c_k| w_q(k) \rho^{|k|}<\infty
\;\forall \rho\in (0,r) \Bigr\},
\end{equation*}
where the function $w_q\colon \Z_+^n\to\R_+$ is given by \eqref{w_q}.
The topology on $\cO_q(\DD^n_r)$ is given by the norms $\|\cdot\|_{\DD,\rho}\; (0<\rho<r)$,
and the multiplication on $\cO_q(\DD^n_r)$ is uniquely determined by $x_i x_j=qx_j x_i$ for all $i<j$.
\end{definition}

In other words, $\cO_q(\DD^n_r)$ is the completion of $\cO_q^\reg(\CC^n)$
with respect to the family $\{\|\cdot\|_{\DD,\rho} : \rho\in (0,r)\}$ of submultiplicative norms.
Letting $q=1$ in Definition~\ref{def:q_poly} and comparing with~\eqref{poly_power_rep}, we see that
the map $\cO(\DD^n_r)\to\cO_1(\DD^n_r)$ taking each $f\in\cO(\DD^n_r)$ to its Taylor
series at $0$ is a Fr\'echet algebra isomorphism.
Note that, if $r=\infty$, then $\cO_q(\DD_r^n)=\cO_q(\CC^n)$
(see \eqref{O_q}).

To define the algebra of holomorphic functions on the quantum ball, we need
the following generalization of~\eqref{poly_power_rep} due to
L.\,A.\,Aizenberg and B.\,S.\,Mityagin \cite{Aiz_Mit}
(see also \cite{Rolewicz}). Given a complete bounded Reinhardt domain
$D\subset\CC^n$, let
\[
b_k(D)=\sup_{z\in D} |z^k|=\sup_{z\in\partial D} |z^k| \qquad (k\in\Z_+^n).
\]
Aizenberg and Mityagin proved that there is a topological isomorphism
\begin{equation}
\label{AizMit}
\cO(D)\cong
\Bigl\{ f=\sum_{k\in\Z_+^n} c_k z^k :
\| f\|_{s}=\sum_{k\in\Z_+^n}
|c_k| b_k(D) s^{|k|}<\infty
\; \forall s\in (0,1)\Bigr\}
\end{equation}
taking each holomorphic function on
$D$ to its Taylor series at $0$.

We clearly have $b_k(\DD^n_r)=r^{|k|}$, so \eqref{poly_power_rep} is
a special case of~\eqref{AizMit}. Consider now the case $D=\BB^n_r$.

\begin{lemma}
\label{lemma:bk_ball}
For each $r\in (0,+\infty)$, we have
\[
b_k(\BB_r^n)=\left(\frac{k^k}{|k|^{|k|}}\right)^{\frac{1}{2}} r^{|k|}.
\]
\end{lemma}
\begin{proof}
This is an elementary calculation involving Lagrange multipliers.
\end{proof}

\begin{corollary}
\label{cor:ball_pow_rep0}
For each $r\in (0,+\infty]$, there is a topological isomorphism
\begin{equation}
\label{ball_pow_rep2}
\cO(\BB_r^n)\cong
\biggl\{ f=\sum_{k\in\Z_+^n} c_k z^k :
\| f\|'_{\BB,\rho}=\sum_{k\in\Z_+^n}
|c_k| \left(\frac{k^k}{|k|^{|k|}}\right)^{\frac{1}{2}} \rho^{|k|}<\infty
\; \forall \rho\in (0,r)\biggr\}
\end{equation}
taking each holomorphic function on $\BB_r^n$ to its Taylor series at $0$.
\end{corollary}
\begin{proof}
Let $\cO'(\BB_r^n)$ denote the power series space on the
right-hand side of~\eqref{ball_pow_rep2}.
If $r<\infty$, then the isomorphism $\cO(\BB^n_r)\cong\cO'(\BB^n_r)$
is immediate from \eqref{AizMit} and Lemma~\ref{lemma:bk_ball}.
To prove the result for $r=\infty$, observe that
we have obvious topological isomorphisms
\begin{equation}
\label{ball_invlim}
\cO(\CC^n)\cong\varprojlim\nolimits_{r<\infty} \cO(\BB_r^n),
\quad \cO'(\CC^n)\cong\varprojlim\nolimits_{r<\infty} \cO'(\BB_r^n).
\end{equation}
Moreover, the isomorphisms $\varphi_r\colon\cO(\BB_r^n)\to\cO'(\BB_r^n)$
(see~\eqref{ball_pow_rep2}), which are already defined for all $r<\infty$,
are clearly compatible with the linking maps of the inverse systems~\eqref{ball_invlim}.
Letting $\varphi_\infty=\varprojlim_{r<\infty} \varphi_r$, we obtain the required topological
isomorphism for $r=\infty$.
\end{proof}

For our purposes, a slightly different power series representation of $\cO(\BB^n_r)$ is needed.
Let us start with an elementary lemma.

\begin{lemma}
\label{lemma:Stirling}
For each $k\in\Z_+^n$, let
\begin{equation}
\label{ak_bk}
a_k=\frac{k!}{|k|!},\quad b_k=\frac{k^k}{|k|^{|k|}}.
\end{equation}
Then
\begin{equation}
\label{asymp}
\lim_{k\to\infty} \left(\frac{a_k}{b_k}\right)^{\frac{1}{|k|}}=1.
\end{equation}
\end{lemma}
\begin{proof}
For each $m\in\Z_+$, let $m^+=m+1$. By Stirling's formula, there exists a
function $\theta\colon\Z_+\to\R$ such that
$\theta(m)\to 0$ as $m\to\infty$ and\footnote{We use $m^+$ instead of $m$ in \eqref{Stirling}
in order to cover the case $m=0$, which is essential when we pass to multi-indices.}
\begin{equation}
\label{Stirling}
m!=\sqrt{2\pi m^+}\, m^m  e^{-m+\theta(m)} \qquad (m\in\Z_+).
\end{equation}
Therefore for each $k=(k_1,\ldots ,k_n)\in\Z_+^n$ we have
\begin{align*}
k! &= (2\pi)^{n/2} (k_1^+\cdots k_n^+)^{1/2} k^k e^{-|k|+\sum_i\theta(k_i)},\\
|k|! &= (2\pi)^{1/2} (|k|+1)^{1/2} |k|^{|k|} e^{-|k|+\theta(|k|)},
\end{align*}
whence
\[
\frac{a_k}{b_k}= (2\pi)^{\frac{n-1}{2}}
\left(\frac{k_1^+\cdots k_n^+}{|k|+1}\right)^{\frac{1}{2}} e^{\tau(k)} \qquad (k\in\Z_+^n),
\]
where $\tau(k)=\sum_i\theta(k_i)-\theta(|k|)$ is bounded. Now, in order to prove~\eqref{asymp},
it remains to show that
\begin{equation}
\label{asymp2}
\lim_{k\to\infty} \left(\frac{k_1^+\cdots k_n^+}{|k|+1}\right)^{\frac{1}{2|k|}}=1.
\end{equation}
We have
\begin{gather*}
\left(\frac{k_1^+\cdots k_n^+}{|k|+1}\right)^{\frac{1}{2|k|}}
\le \left(\frac{(|k|+1)^n}{|k|+1}\right)^{\frac{1}{2|k|}}
=(|k|+1)^{\frac{n-1}{2|k|}} \to 1 \qquad (k\to\infty);\\
\left(\frac{k_1^+\cdots k_n^+}{|k|+1}\right)^{\frac{1}{2|k|}}
\ge \left(\frac{\max_i k_i^+}{|k|+1}\right)^{\frac{1}{2|k|}}
\ge \left(\frac{1}{n}\right)^{\frac{1}{2|k|}} \to 1 \qquad (k\to\infty).
\end{gather*}
This proves \eqref{asymp2}, which, in turn, implies \eqref{asymp}.
\end{proof}

\begin{prop}
\label{prop:ball_pow_ser}
For each $r\in (0,+\infty]$, there is a topological isomorphism
\begin{equation}
\label{ball_pow_rep}
\cO(\BB_r^n)\cong
\biggl\{ f=\sum_{k\in\Z_+^n} c_k z^k :
\| f\|_{\BB,\rho}=\sum_{k\in\Z_+^n}
|c_k| \left(\frac{k!}{|k|!}\right)^{\frac{1}{2}} \rho^{|k|}<\infty
\; \forall \rho\in (0,r)\biggr\}
\end{equation}
taking each holomorphic function on $\BB_r^n$ to its Taylor series at $0$.
\end{prop}
\begin{proof}
In view of Corollary~\ref{cor:ball_pow_rep0},
it suffices to show that the families
\begin{equation}
\label{fam_norms_ball}
\{ \|\cdot\|_{\BB,\rho} : \rho\in (0,r)\}\quad\text{and}\quad
\{ \|\cdot\|_{\BB,\rho}' : \rho\in (0,r)\}
\end{equation}
of norms are equivalent on $\CC[z_1,\ldots ,z_n]$. Define sequences $(a_k)_{k\in\Z_+^n}$
and $(b_k)_{k\in\Z_+^n}$ by~\eqref{ak_bk}. Fix any $\rho\in (0,r)$, and choose
$\rho_1\in (\rho,r)$. By Lemma~\ref{lemma:Stirling}, there exists $K\in\N$ such that
\begin{equation}
\label{ak_bk_est}
\frac{\rho}{\rho_1} \le \left(\frac{a_k}{b_k}\right)^{\frac{1}{2|k|}} \le \frac{\rho_1}{\rho}
\qquad (|k|\ge K).
\end{equation}
Using \eqref{ak_bk_est}, we can find $C>0$ such that
\[
a_k^{1/2} \le C b_k^{1/2} \left(\frac{\rho_1}{\rho}\right)^{|k|},
\quad b_k^{1/2} \le C a_k^{1/2} \left(\frac{\rho_1}{\rho}\right)^{|k|}
\qquad (k\in\Z_+^n).
\]
Hence for each $f=\sum_k c_k z^k\in\CC[z_1,\ldots ,z_n]$ we have
\begin{gather*}
\| f\|_{\BB,\rho}=\sum_k |c_k| a_k^{1/2} \rho^{|k|}
\le C\sum_k |c_k| b_k^{1/2} \rho_1^{|k|} = C\| f\|'_{\BB,\rho_1},\\
\| f\|'_{\BB,\rho}=\sum_k |c_k| b_k^{1/2} \rho^{|k|}
\le C\sum_k |c_k| a_k^{1/2} \rho_1^{|k|} = C\| f\|_{\BB,\rho_1}.
\end{gather*}
Thus the families \eqref{fam_norms_ball} of norms on $\CC[z_1,\ldots ,z_n]$ are equivalent,
and hence the power series spaces on the right-hand sides of~\eqref{ball_pow_rep}
and~\eqref{ball_pow_rep2} coincide. This completes the proof.
\end{proof}

\begin{remark}
At the moment, it is not obvious whether the norms $\|\cdot\|_{\BB,\rho}$
defined by~\eqref{ball_pow_rep} are submultiplicative on $\cO(\BB_r^n)$. In fact they are, and
this can be proved directly by using the inequality
\[
\binom{m}n\binom{p}q\le\binom{m+p}{n+q},
\]
which is immediate from the Chu-Vandermonde formula (see, e.g., \cite[1.1.17]{Stanley2}).
We omit the details, because a more general result will be proved
in Theorem~\ref{thm:qball}.
\end{remark}

Now, in order to define a $q$-analog of $\cO(\BB^n_r)$, we need to
``quantize'' the norms $\|\cdot\|_{\BB,\rho}$ given by~\eqref{ball_pow_rep}.
This will be done in the following two lemmas.

\begin{lemma}
For each $q>0$ and for each $k,\ell\in\Z_+^n$, we have
\begin{equation}
\label{q-CV-ineq}
\frac{\bigl[|k+\ell|\bigr]_q!}{[k+\ell]_q!} \ge \frac{\bigl[|k|\bigr]_q!}{[k]_q!}
\frac{\bigl[|\ell|\bigr]_q!}{[\ell]_q!}\, q^{\sum_{i<j} k_i \ell_j}.
\end{equation}
\end{lemma}
\begin{proof}
We use induction on $n$. For $n=2$, the $q$-analog of the Chu-Vandermonde formula
(see, e.g., \cite[2.1.2, Proposition 3]{Klim_Schm}) implies that
\[
\begin{split}
\binom{k_1+\ell_1+k_2+\ell_2}{k_1+\ell_1}_q
&=\sum_{j=0}^{k_1+k_2} \binom{k_1+k_2}{j}_q \binom{\ell_1+\ell_2}{k_1+\ell_1-j}_q
q^{j(\ell_2-k_1+j)}\\
&\ge \binom{k_1+k_2}{k_1}_q \binom{\ell_1+\ell_2}{\ell_1}_q q^{k_1 \ell_2}.
\end{split}
\]
This is exactly \eqref{q-CV-ineq} for $n=2$. Suppose now that $n\ge 3$, and, for each
$k=(k_1,\ldots ,k_n)\in\Z_+^n$, let $k'=(k_1,\ldots ,k_{n-1})\in\Z_+^{n-1}$.
By the induction hypothesis, we have
\begin{equation}
\label{CV_ind_n-1}
\frac{\bigl[|k'+\ell'|\bigr]_q!}{[k'+\ell']_q!} \ge \frac{\bigl[|k'|\bigr]_q!}{[k']_q!}
\frac{\bigl[|\ell'|\bigr]_q!}{[\ell']_q!}\, q^{\sum_{i<j\le n-1} k_i \ell_j}.
\end{equation}
Applying \eqref{q-CV-ineq} to the $2$-tuples $(|k'|,k_n)$ and $(|\ell'|,\ell_n)$, we get
\begin{equation}
\label{CV_ind_2}
\frac{\bigl[|k'|+k_n+|\ell'|+\ell_n\bigr]_q!}{\bigl[|k'|+|\ell'|\bigr]_q!\, [k_n+\ell_n]_q!}
\ge \frac{\bigl[|k'|+k_n\bigr]_q!}{\bigl[|k'|\bigr]_q!\, [k_n]_q!}
\frac{\bigl[|\ell'|+\ell_n\bigr]_q!}{\bigl[|\ell'|\bigr]_q!\, [\ell_n]_q!}\, q^{|k'|\ell_n}.
\end{equation}
Multiplying \eqref{CV_ind_n-1} by \eqref{CV_ind_2} yields \eqref{q-CV-ineq}.
\end{proof}

\begin{lemma}
\label{lemma:qball_submult}
Given $q\in\CC^\times$, let
\[
u_q(k)=|q|^{\sum_{i<j} k_i k_j} \qquad (k=(k_1,\ldots ,k_n)\in\Z_+^n).
\]
For each $\rho>0$, define a norm on $\cO_q^\reg(\CC^n)$ by
\begin{equation*}
\| f\|_{\BB,\rho} = \sum_{k\in\Z_+^n}
|c_k| \left(\frac{[k]_{|q|^2}!}{\bigl[ |k|\bigr]_{|q|^2}!}\right)^{1/2}\!\!\!\!
u_q(k) \rho^{|k|}
\qquad \Bigl(f=\sum_{k\in\Z_+^n} c_k x^k\in\cO_q^\reg(\CC^n)\Bigr).
\end{equation*}
Then $\|\cdot\|_{\BB,\rho}$ is submultiplicative.
\end{lemma}
\begin{proof}
For each $k,\ell\in\Z_+^n$, let
\[
c(k,\ell)=q^{-\sum_{i>j} k_i \ell_j}.
\]
Then
\[
x^k x^\ell =c(k,\ell) x^{k+\ell}
\]
(see, e.g., \cite{Art_poly,Good_Letz}). By \cite[Lemma 5.7]{Pir_qfree},
$\|\cdot\|_{\BB,\rho}$ is submultiplicative if and only if for each $k,\ell\in\Z_+^n$
\begin{equation}
\label{qball_submult}
\left(\frac{[k+\ell]_{|q|^2}!}{\bigl[ |k+\ell|\bigr]_{|q|^2}!}\right)^{1/2}\!\!\!\! u_q(k+\ell)\, |c(k,\ell)|
\le \left(\frac{[k]_{|q|^2}!\, [\ell]_{|q|^2}!}%
{\bigl[ |k|\bigr]_{|q|^2}!\, \bigl[ |\ell|\bigr]_{|q|^2}!}\right)^{1/2}\!\!\!\! u_q(k) u_q(\ell).
\end{equation}
We have
\begin{align*}
u_q(k+\ell)\, |c(k,\ell)| &= |q|^{\sum_{i<j}(k_i+\ell_i)(k_j+\ell_j)} |q|^{-\sum_{i<j} k_j \ell_i}
=|q|^{\sum_{i<j} k_i k_j +\ell_i\ell_j+k_i\ell_j};\\
u_q(k) u_q(\ell)&=|q|^{\sum_{i<j} k_i k_j +\ell_i\ell_j}.
\end{align*}
Therefore \eqref{qball_submult} is equivalent to
\begin{equation}
\label{qball_submult_2}
\left(\frac{[k+\ell]_{|q|^2}!}{\bigl[ |k+\ell|\bigr]_{|q|^2}!}\right)^{1/2} \le
\left(\frac{[k]_{|q|^2}!\, [\ell]_{|q|^2}!}%
{\bigl[ |k|\bigr]_{|q|^2}!\, \bigl[ |\ell|\bigr]_{|q|^2}!}\right)^{1/2}\!\!\!\!
|q|^{-\sum_{i<j} k_i\ell_j}.
\end{equation}
Raising \eqref{qball_submult_2} to the power $-2$ yields~\eqref{q-CV-ineq} with $q$ replaced
by $|q|^2$. The rest is clear.
\end{proof}

\begin{theorem}
\label{thm:qball}
Let $q\in\CC^\times$, and let $r\in (0,+\infty]$. The Fr\'echet space
\begin{equation*}
\cO_q(\BB_r^n)=
\biggl\{ f=\sum_{k\in\Z_+^n} c_k x^k :
\| f\|_{\BB,\rho}=\sum_{k\in\Z_+^n}
|c_k| \left(\frac{[k]_{|q|^2}!}{\bigl[ |k|\bigr]_{|q|^2}!}\right)^{1/2}\!\!\!\!
u_q(k) \rho^{|k|}<\infty\; \forall \rho\in (0,r)\biggr\}
\end{equation*}
is a Fr\'echet-Arens-Michael algebra with respect to the multiplication uniquely determined
by $x_i x_j=q x_j x_i\; (i<j)$. Moreover, the norms $\|\cdot\|_{\BB,\rho}$
are submultiplicative on $\cO_q(\BB^n_r)$.
\end{theorem}
\begin{proof}
Immediate from Lemma~\ref{lemma:qball_submult}.
\end{proof}

\begin{definition}
\label{def:qball}
The Fr\'echet algebra $\cO_q(\BB^n_r)$ will be called
the {\em algebra of holomorphic functions on the quantum $n$-ball of radius $r$}.
\end{definition}

In other words, $\cO_q(\BB^n_r)$ is the completion of $\cO_q^\reg(\CC^n)$
with respect to the family $\{\|\cdot\|_{\BB,\rho} : \rho\in (0,r)\}$ of submultiplicative norms.
Letting $q=1$ in Definition~\ref{def:qball} and comparing with~\eqref{ball_pow_rep}, we see that
the map $\cO(\BB^n_r)\to\cO_1(\BB^n_r)$ taking each $f\in\cO(\BB^n_r)$ to its Taylor
series at $0$ is a Fr\'echet algebra isomorphism.

\begin{remark}
At the moment, we still do not know whether $\cO_q(\BB^n_\infty)=\cO_q(\CC^n)$.
This will be proved in Corollary~\ref{cor:B_inf=C^n}.
\end{remark}

\begin{remark}
\label{rem:no_mult}
One may wonder why we do not try to define $q$-deformed multiplications on the same
Fr\'echet spaces $\cO(\DD^n_r)$ and $\cO(\BB^n_r)$. The answer is that such multiplications
do not exist in general. To be more precise, if $|q|<1$ and $n\ge 2$,
then there is no continuous multiplication $\star$
on $\cO(\DD^n_r)$ such that $x_i x_j=x_i\star x_j=qx_j \star x_i$ for all $i<j$
(where $x_1,\ldots ,x_n$ are the coordinates on $\CC^n$).
Indeed, assume that such a multiplication exists. Then for each $\rho\in (0,r)$ there exist
$C>0$ and $s\in (0,r)$ such that $\| f\star g\|_\rho\le C\| f\|_s \| g\|_s$
($f,g\in\cO(\DD^n_r)$). In particular, for each $m\in\N$ we have
$\| x_2^m \star x_1^m\|_\rho\le C\| x_2^m\|_s \| x_1^m\|_s=Cs^{2m}$.
Since $\| x_2^m \star x_1^m\|_\rho=\| q^{-m^2} x_1^m x_2^m\|_\rho=|q|^{-m^2}\rho^{2m}$,
we conclude that $|q|^{-m^2}\le C(s/\rho)^{2m}$ for all $m$, which is a contradiction.
A similar argument works for $\cO(\BB^n_r)$.
\end{remark}

Our next goal is to extend the algebra isomorphism
\begin{equation}
\label{tau_q_qinv}
\tau\colon\cO_q^\reg(\CC^n)\to\cO_{q^{-1}}^\reg(\CC^n),\quad x_i\mapsto x_{n-i},
\end{equation}
to $\cO_q(\DD^n_r)$ and $\cO_q(\BB^n_r)$. To this end, we need a lemma.

\begin{lemma}
\label{lemma:binom_q_qinv}
For each $q>0$ and each $k=(k_1,\ldots ,k_n)\in\Z_+^n$, we have
\[
\frac{[k]_q!}{\bigl[ |k|\bigr]_q!}\, q^{\sum_{i<j} k_i k_j}
= \frac{[k]_{q^{-1}}!}{\bigl[ |k|\bigr]_{q^{-1}}!}.
\]
\end{lemma}
\begin{proof}
For each $m\in\Z_+$, we have $[m]_q=q^{m-1}[m]_{q^{-1}}$, and so
\[
[m]_q!=q^{\tfrac{m(m-1)}{2}} [m]_{q^{-1}}!.
\]
Therefore
\[
\begin{split}
\frac{[k]_q!}{\bigl[ |k|\bigr]_q!}\, q^{\sum_{i<j} k_i k_j}
&=\frac{q^{\sum_i \tfrac{k_i(k_i-1)}{2}} [k]_{q^{-1}}!}%
{q^{\tfrac{|k|(|k|-1)}{2}} \bigl[|k|\bigr]_{q^{-1}}!}\, q^{\sum_{i<j} k_i k_j}\\
&=\frac{q^{\tfrac{\sum_i k_i^2-\sum_i k_i}{2}} [k]_{q^{-1}}!}%
{q^{\tfrac{\sum_i k_i^2+2\sum_{i<j}k_i k_j-\sum_i k_i}{2}} \bigl[|k|\bigr]_{q^{-1}}!}%
\, q^{\sum_{i<j} k_i k_j}
=\frac{[k]_{q^{-1}}!}{\bigl[ |k|\bigr]_{q^{-1}}!}. \qedhere
\end{split}
\]
\end{proof}

\begin{corollary}
\label{cor:qball_norms_2}
For each $q\in\CC^\times$, each $f\in\cO_q(\BB^n_r)$, and each $\rho\in (0,r)$, we have
\[
\| f\|_{\BB,\rho}=\sum_{k\in\Z_+^n}
|c_k| \left(\frac{[k]_{|q|^{-2}}!}{\bigl[ |k|\bigr]_{|q|^{-2}}!}\right)^{1/2}\!\!\!\!
\rho^{|k|}.
\]
\end{corollary}
\begin{proof}
Apply Lemma \ref{lemma:binom_q_qinv} with $q$ replaced by $|q|^2$.
\end{proof}

\begin{prop}
\label{prop:q-qinv}
For each $q\in\CC^\times$ and each $r\in (0,+\infty]$, there exist topological algebra isomorphisms
\begin{gather*}
\tau_\DD\colon\cO_q(\DD^n_r)\to\cO_{q^{-1}}(\DD^n_r),\quad x_i\mapsto x_{n-i};\\
\tau_\BB\colon\cO_q(\BB^n_r)\to\cO_{q^{-1}}(\BB^n_r),\quad x_i\mapsto x_{n-i}.
\end{gather*}
Moreover, for each $f\in\cO_q(\DD^n_r)$, each $g\in\cO_q(\BB^n_r)$, and each $\rho\in (0,r)$, we have
\begin{equation}
\label{q-qinv_norms}
\| \tau_\DD(f)\|_{\DD,\rho}=\| f\|_{\DD,\rho};\quad
\| \tau_\BB(g)\|_{\BB,\rho}=\| g\|_{\BB,\rho}.
\end{equation}
\end{prop}
\begin{proof}
For convenience, let us denote the norm $\|\cdot\|_{\DD,\rho}$
on $\cO_q(\DD^n_r)$ by $\|\cdot\|_{\DD,q,\rho}$. Similarly, we write $\|\cdot\|_{\BB,q,\rho}$
for the norm $\|\cdot\|_{\BB,\rho}$ on $\cO_q(\BB^n_r)$.
In view of~\eqref{tau_q_qinv}, it suffices to show that
for each $f\in\cO_q^\reg(\CC^n)$ and each $\rho>0$, we have
\begin{align}
\label{tau_isometr_D}
\| \tau(f)\|_{\DD,q^{-1},\rho}&=\| f\|_{\DD,q,\rho};\\
\label{tau_isometr_B}
\| \tau(f)\|_{\BB,q^{-1},\rho}&=\| f\|_{\BB,q,\rho}.
\end{align}
Without loss of generality, we may assume that $|q|\le 1$.

Observe that, for each $k=(k_1,\ldots ,k_n)\in\Z_+^n$, we have
\[
\tau(x^k)=x_n^{k_1}\cdots x_1^{k_n}=q^{\sum_{i<j} k_i k_j} x_1^{k_n}\cdots x_n^{k_1}.
\]
Therefore, for each $f=\sum_k c_k x^k\in\cO_q^\reg(\CC^n)$,
\[
\| \tau(f)\|_{\DD,q^{-1},\rho}
=\sum_k |c_k| w_q(k) \rho^{|k|} = \| f\|_{\DD,q,\rho}.
\]
This proves \eqref{tau_isometr_D}.
Similarly, using Corollary~\ref{cor:qball_norms_2}, we obtain
\[
\| \tau(f)\|_{\BB,q^{-1},\rho}
=\sum_k |c_k| u_q(k) \left(\frac{[k]_{|q|^2}!}{\bigl[ |k|\bigr]_{|q|^2}!}\right)^{1/2}%
\!\!\!\!\rho^{|k|}=\| f\|_{\BB,q,\rho}.
\]
This proves \eqref{tau_isometr_B} and completes the proof.
\end{proof}

\section{Quantum ball \`a la Vaksman}
\label{sect:qball_Vaks}

In this section we establish a relation between $\cO_q(\BB^n_r)$ and
L.\,L.\,Vaksman's algebra $C_q(\bar\BB^n)$, which is a natural
$q$-analog of the algebra $C(\bar\BB^n)$ of continuous functions on
the closed unit ball $\bar\BB^n=\bar\BB^n_1$~\cite{Vaks_max}.
To motivate the construction, let us start with the classical situation.
Let $\Fun(\CC^n)$ denote the algebra of all $\CC$-valued functions on $\CC^n$.
There is a natural involution on $\Fun(\CC^n)$ given by $f^*(z)=\ol{f(z)}$.
Let $\Pol(\CC^n)$ denote the $*$-subalgebra of $\Fun(\CC^n)$
generated by the coordinates $z_1,\ldots ,z_n$ on $\CC^n$.
Clearly, we have an algebra isomorphism
$\Pol(\CC^n)\cong\CC[z_1,\ldots ,z_n,\bar z_1,\ldots ,\bar z_n]$.
By the Stone-Weierstrass theorem, the completion of $\Pol(\CC^n)$ with respect
to the uniform norm $\| f\|_\BB^\infty=\sup_{z\in\bar\BB^n} |f(z)|$ is
$C(\bar\BB^n)$.
For each $\rho>0$ and each $f\in \CC[z_1,\ldots ,z_n]$, let
\begin{equation*}
\| f\|_{\BB,\rho}^\infty=\sup_{z\in\bar\BB^n_\rho} |f(z)| = \|\gamma_\rho(f)\|_\BB^\infty,
\end{equation*}
where $\gamma_\rho$ is the automorphism of $\CC[z_1,\ldots ,z_n]$ uniquely
determined by $\gamma_\rho(z_i)=\rho z_i$  ($i=1,\ldots ,n$).
Clearly, the completion of $\CC[z_1,\ldots ,z_n]$ with respect to
the family $\{\|\,\cdot\,\|_{\BB,\rho} : \rho\in (0,r)\}$ of norms is topologically isomorphic
to $\cO(\BB^n_r)$.

Now let us ``quantize'' the above data.
Fix $q\in (0,1)$, and denote by $\Pol_q(\CC^n)$
the $*$-algebra generated (as a $*$-algebra) by $n$ elements $x_1,\ldots ,x_n$
subject to the relations
\begin{equation}
\label{tw_CCR}
\begin{aligned}
x_i x_j&=q x_j x_i \quad (i<j);\\
x_i^* x_j&= q x_j x_i^* \quad (i\ne j);\\
x_i^* x_i&=q^2 x_i x_i^*+(1-q^2)\Bigl(1-\sum_{k>i} x_k x_k^*\Bigr).
\end{aligned}
\end{equation}
Clearly, for $q=1$ we have $\Pol_q(\CC^n)\cong\Pol(\CC^n)$.
The algebra $\Pol_q(\CC^n)$ was introduced by W.\,Pusz and S.\,L.\,Woronowicz \cite{PW},
although they used different $*$-generators $a_1,\ldots ,a_n$ given by
$a_i=(1-q^2)^{-1/2} x_i^*$. Relations \eqref{tw_CCR} divided by $1-q^2$ and
written in terms of the $a_i$'s
are called the ``twisted canonical commutation relations'', and the algebra $A_q=\Pol_q(\CC^n)$
defined in terms of the $a_i$'s is sometimes called the ``quantum Weyl algebra''
(see, e.g., \cite{WZ,Alev,Jordan,Klim_Schm}).
Note that, while $\Pol_q(\CC^n)$ becomes $\Pol(\CC^n)$ for $q=1$, $A_q$ becomes
the Weyl algebra. The idea to use the generators $x_i$
instead of the $a_i$'s and to consider $\Pol_q(\CC^n)$ as a $q$-analog of $\Pol(\CC^n)$
is probably due to Vaksman \cite{Vaks_splet}; the one-dimensional case
was considered in \cite{Klim_Lesn_disk}.
The algebra $\Pol_q(\CC^n)$ serves as a basic example in the general theory
of quantum bounded symmetric domains developed by Vaksman and his
collaborators (see \cite{Vaks_sborn,Vaks_book} and references therein).

Let $H$ be a Hilbert space with an orthonormal basis
$\{ e_k : k\in \Z_+^n\}$. Following \cite{PW}, for each
$k=(k_1,\ldots ,k_n)\in\Z_+^n$ we will write $|k_1,\ldots ,k_n\ra$
for $e_k$. As was proved by Pusz and Woronowicz \cite{PW}, there exists
a faithful irreducible $*$-representation $\pi$ of $\Pol_q(\CC^n)$ on $H$ uniquely
determined by
\begin{gather*}
\pi(x_j)e_k=\sqrt{1-q^2} \sqrt{[k_j+1]_{q^2}}\, q^{\sum_{i>j}k_i}
|k_1,\ldots ,k_j+1,\ldots ,k_n\ra\\
(j=1,\ldots ,n,\; k=(k_1,\ldots ,k_n)\in\Z_+^n).
\end{gather*}
The completion of $\Pol_q(\CC^n)$ with respect to the operator norm $\| a\|_\op=\| \pi(a)\|$
is denoted by $C_q(\bar\BB^n)$ and is called the {\em algebra of continuous functions
on the closed quantum ball} \cite{Vaks_max}; see also \cite{PW,Prosk_Sam}.

Observe now that the subalgebra of $\Pol_q(\CC^n)$ generated by $x_1,\ldots ,x_n$
is exactly $\cO_q^\reg(\CC^n)$. For each $\rho>0$, let $\gamma_\rho$ be the automorphism
of $\cO_q^\reg(\CC^n)$ uniquely
determined by $\gamma_\rho(x_i)=\rho x_i\; (i=1,\ldots ,n)$.
Define a submultiplicative norm $\|\cdot\|_{\BB,\rho}^\infty$ on
$\cO_q^\reg(\CC^n)$ by
\begin{equation}
\label{norms_qball_op}
\| a\|_{\BB,\rho}^\infty=\|\gamma_\rho(a)\|_\op \qquad (a\in\cO_q^\reg(\CC^n)).
\end{equation}
The completion of $\cO_q^\reg(\CC^n)$ with respect to the family
$\{\|\cdot\|_{\BB,\rho}^\infty : \rho\in (0,r)\}$
of norms will be denoted by $\cO_q^\rV(\BB^n_r)$
(the superscript ``$\rV$'' is for ``Vaksman'').
It follows from the discussion at the beginning of this section that
$\cO_q^\rV(\BB^n_r)$ is indeed a natural $q$-analog of $\cO(\BB^n_r)$.

The main result of this section is the following theorem.

\begin{theorem}
\label{thm:poly_vs_ball}
For each $q\in (0,1)$ and each $r\in (0,+\infty]$,
there exists a topological algebra isomorphism
\[
\cO_q^\rV(\BB^n_r)\lriso\cO_q(\BB^n_r), \quad x_i\mapsto x_i\quad (i=1,\ldots ,n).
\]
\end{theorem}

The proof of Theorem~\ref{thm:poly_vs_ball} will be divided into several lemmas.

\begin{lemma}
\label{lemma:vacuum}
For each $k\in\Z_+^n$, we have
\begin{equation*}
\pi(x^k)e_0=\sqrt{[k]_{q^2}!}\, (1-q^2)^{\tfrac{|k|}{2}} w_q(k) e_k.
\end{equation*}
\end{lemma}
\begin{proof}
We use induction on $|k|$. For $|k|=0$ there is nothing to prove.
Suppose now that $|k|>0$, and let $m=\min\{ i=1,\ldots ,n : k_i\ne 0\}$.
We have
\[
x^k=x_m  x^\ell, \quad\text{where}\quad\ell=(0,\ldots ,0,k_m-1,k_{m+1},\ldots ,k_n).
\]
Using the induction hypothesis, we obtain
\[
\begin{split}
\pi(x^k)e_0
&= \sqrt{[\ell]_{q^2}!}\, (1-q^2)^{\tfrac{|k|-1}{2}}
q^{\sum_{m+1\le i<j} k_i k_j+(k_m-1)\sum_{j>m} k_j}\,
\pi(x_m) e_\ell\\
&=\sqrt{[\ell]_{q^2}!}\, (1-q^2)^{\tfrac{|k|-1}{2}}
q^{\sum_{m+1\le i<j} k_i k_j+(k_m-1)\sum_{j>m} k_j}
\sqrt{1-q^2} \sqrt{[k_m]_{q^2}}\, q^{\sum_{j>m}k_j}\, e_k\\
&=\sqrt{[k]_{q^2}!}\, (1-q^2)^{\tfrac{|k|}{2}} q^{\sum_{i<j} k_i k_j}\, e_k. \qedhere
\end{split}
\]
\end{proof}

It will be convenient to introduce one more family of norms on $\cO_q^\reg(\CC^n)$.
Namely, for each $\rho>0$ we let
\begin{equation*}
\| a\|_{\DD,\rho}^{(2)}=\Bigl(\sum_{k\in\Z_+^n} |c_k|^2 w_q^2(k) \rho^{2|k|}\Bigr)^{1/2}
\qquad \Bigl( a=\sum_{k\in\Z_+^n} c_k x^k \in \cO_q^\reg(\CC^n) \Bigr).
\end{equation*}

\begin{lemma}
\label{lemma:l1-l2_equiv}
For each $0<\rho<\tau<+\infty$ we have
\[
\| \cdot \|_{\DD,\rho}^{(2)} \le \| \cdot \|_{\DD,\rho}
\le \left(\frac{\tau^2}{\tau^2-\rho^2}\right)^{n/2} \!\!\| \cdot \|_{\DD,\tau}^{(2)}.
\]
\end{lemma}
\begin{proof}
The first inequality is obvious (the $\ell^2$-norm is less than or equal to the $\ell^1$-norm).
The second one is a simple application
of the Cauchy-Bunyakowsky-Schwarz inequality:
\[
\begin{split}
\| a\|_{\DD,\rho}
&=\sum_{k\in\Z_+^n} |c_k| w_q(k) \tau^{|k|} (\rho/\tau)^{|k|}\\
&\le \Bigl(\sum_{k\in\Z_+^n} |c_k|^2 w_q^2(k) \tau^{2|k|}\Bigr)^{1/2}
\Bigl(\sum_{k\in\Z_+^n} (\rho/\tau)^{2|k|}\Bigr)^{1/2}
=\left(\frac{\tau^2}{\tau^2-\rho^2}\right)^{n/2} \!\!\| a \|_{\DD,\tau}^{(2)}. \qedhere
\end{split}
\]
\end{proof}

\begin{lemma}
\label{lemma:op-l1_equiv}
For each $\rho>0$, we have
\begin{equation}
\label{op-l1_equiv}
\| \cdot \|_{\BB,\rho}^\infty \le \| \cdot\|_{\DD,\rho}\, , \quad
\| \cdot \|_{\BB,\rho}^\infty\ge (q^2;q^2)_\infty^{n/2} \| \cdot \|_{\DD,\rho}^{(2)}.
\end{equation}
\end{lemma}
\begin{proof}
Let us first consider the case $\rho=1$, so that $\|\cdot\|_{\BB,1}^\infty=\|\cdot\|_\op$.
By \cite{PW}, we have\footnote{In fact, it is easy to show that
$\| x_i\|_\op = 1$ for all $i=1,\ldots ,n$, but we will not use this equality here.}
$\| x_i\|_\op \le 1$ for all $i=1,\ldots ,n$. Using the maximality property of $\|\cdot\|_{\DD,1}$
(see~\cite[Lemma 5.10]{Pir_qfree}), we conclude that $\|\cdot\|_\op\le \|\cdot\|_{\DD,1}$.

Now take any $a=\sum_k c_k x^k\in\cO_q^\reg(\CC^n)$. Using Lemma~\ref{lemma:vacuum},
we see that
\begin{equation}
\label{opnorm_est}
\| a\|_{\op}^2 \ge \| \pi(a)e_0\|^2
=\sum_{k\in\Z_+^n} |c_k|^2 [k]_{q^2}!\, (1-q^2)^{|k|} w_q^2(k).
\end{equation}
Observe that for each $\ell\in\N$
\[
[\ell]_{q^2}!\, (1-q^2)^\ell=\prod_{j=1}^\ell (1-q^{2j}) \ge \prod_{j=1}^\infty (1-q^{2j})
=(q^2;q^2)_\infty,
\]
and so for each $k=(k_1,\ldots ,k_n)\in\Z_+^n$
\[
[k]_{q^2}!\,(1-q^2)^{|k|}=\prod_{i=1}^n [k_i]_{q^2}!\, (1-q^2)^{k_i} \ge (q^2;q^2)_\infty^n.
\]
Now it follows from \eqref{opnorm_est} that
\[
\| a\|_\op^2 \ge (q^2;q^2)_\infty^n \sum_{k\in\Z_+^n} |c_k|^2 w_q^2(k)
= (q^2;q^2)_\infty^n (\| a\|_{\DD,1}^{(2)})^2.
\]
Thus we have proved \eqref{op-l1_equiv} for $\rho=1$. The general case reduces to
the case $\rho=1$ by using~\eqref{norms_qball_op} and by observing that
\[
\| a\|_{\DD,\rho}=\|\gamma_\rho(a)\|_{\DD,1}\, ,\quad
\| a\|_{\DD,\rho}^{(2)}=\|\gamma_\rho(a)\|_{\DD,1}^{(2)}. \qedhere
\]
\end{proof}

\begin{proof}[Proof of Theorem~{\upshape\ref{thm:poly_vs_ball}}]
Applying Lemmas~\ref{lemma:l1-l2_equiv} and \ref{lemma:op-l1_equiv}, we see that
the families
\[
\{ \|\cdot\|_{\DD,\rho} : \rho\in (0,r)\},
\quad \{ \|\cdot\|_{\DD,\rho}^{(2)} : \rho\in (0,r)\},
\quad \{ \|\cdot\|_{\BB,\rho}^\infty : \rho\in (0,r)\}
\]
of norms on $\cO_q^\reg(\CC^n)$ are equivalent. The rest follows from
Theorem~\ref{thm:poly_vs_ball_2}.
\end{proof}

\begin{remark}
In the special case where $r=\infty$, Theorem~\ref{thm:poly_vs_ball} is equivalent
to~\cite[Theorem 5.16]{Pir_qfree}.
\end{remark}

\begin{remark}
Lemmas~\ref{lemma:l1-l2_equiv} and \ref{lemma:op-l1_equiv}
imply that for each $0<\rho<\tau<+\infty$ we have
\begin{equation}
\label{key}
\Bigl(\frac{\tau^2-\rho^2}{\tau^2}\, (q^2;q^2)_\infty\Bigr)^{n/2} \| \cdot\|_{\DD,\rho}
\le \| \cdot\|_{\BB,\tau}^\infty \le \| \cdot\|_{\DD,\tau}.
\end{equation}
While the second inequality in \eqref{key} holds in the classical case $q=1$ as well,
the first inequality in \eqref{key} becomes useless
(since $(q^2;q^2)_\infty\to 0$ as $q\to1$). Geometrically, this can be explained as follows.
If we fix $\tau$ and take $\rho<\tau$ close enough to $\tau$, then the polydisk of
radius $\rho$ will not be contained in the ball of radius $\tau$, and so the supremum over the
polydisk (which is less than or equal to $\|\cdot\|_{\DD,\rho}$)
will not be dominated by the supremum $\|\cdot\|_{\BB,\tau}^\infty$ over the ball.
\end{remark}

\section{A $q$-analog of Poincar\'e's theorem}
\label{sect:Poincare}
A classical result of H.~Poincar\'e \cite{Poincare} (see also \cite[Theorem 2.7]{Range}) asserts that
the polydisk $\DD_r^n$ and the ball $\BB_r^n$ are not biholomorphically equivalent
(unless $n=1$ or $r=\infty$). By O.~Forster's theorem~\cite{For} (see also \cite[V.7]{GR_Stein}),
the category of Stein spaces is anti-equivalent to the category of Stein
algebras (i.e., Fr\'echet algebras of the form $\cO(X)$, where $X$ is a Stein space)
via the functor $X\mapsto\cO(X)$.
Therefore, when translated into the dual language, Poincar\'e's theorem states
that the Fr\'echet algebras $\cO(\DD_r^n)$ and $\cO(\BB_r^n)$ are not topologically isomorphic.

A natural question is whether or not Poincar\'e's theorem has a $q$-analog, i.e., whether or not
the Fr\'echet algebras $\cO_q(\DD^n_r)$ and $\cO_q(\BB^n_r)$ are topologically
isomorphic. The goal of this section is to answer the above question.

\begin{lemma}
\label{lemma:binom_est}
For each $q\in (0,1)$ and each $k\in\Z_+^n$, we have
\[
(q;q)_\infty^n\le\frac{[k]_q!}{\bigl[|k|\bigr]_q!}\le 1.
\]
\end{lemma}
\begin{proof}
Let $k=(k_1,\ldots ,k_n)\in\Z_+^n$. Then
\[
\frac{[k]_q!}{\bigl[|k|\bigr]_q!}
=\frac{\prod_{i=1}^n [k_i]_q!}{[k_1+\cdots +k_n]_q!}
=\frac{\prod_{i=1}^n \prod_{j=1}^{k_i} (1-q^j)}{\prod_{p=1}^{k_1+\cdots +k_n} (1-q^p)}
=\prod_{i=1}^n \prod_{j=1}^{k_i}\frac{1-q^j}{1-q^{k_1+\cdots +k_{i-1}+j}}\le 1.
\]
On the other hand,
\[
\frac{[k]_q!}{\bigl[|k|\bigr]_q!} \ge \prod_{i=1}^n \prod_{j=1}^{k_i}(1-q^j)
\ge \prod_{i=1}^n \prod_{j=1}^\infty(1-q^j)=(q;q)_\infty^n. \qedhere
\]
\end{proof}

\begin{theorem}
\label{thm:poly_vs_ball_2}
Let $q\in \CC^\times,\; |q|\ne 1$, and let $r\in (0,+\infty]$. Then
$\cO_q(\DD^n_r)=\cO_q(\BB^n_r)$
as vector subspaces of $\CC[[x_1,\ldots ,x_n]]$ and as Fr\'echet algebras.
Moreover, for each $\rho\in (0,r)$ we have
\begin{align}
\label{DBD>1}
(|q|^{-2};|q|^{-2})_\infty^{n/2} \|\cdot\|_{\DD,\rho}
&\le\|\cdot\|_{\BB,\rho}\le\|\cdot\|_{\DD,\rho}\qquad (|q|>1);\\
\label{DBD<1}
(|q|^2;|q|^2)_\infty^{n/2} \|\cdot\|_{\DD,\rho}
&\le\|\cdot\|_{\BB,\rho}\le\|\cdot\|_{\DD,\rho}\qquad (|q|<1).
\end{align}
\end{theorem}
\begin{proof}
To prove the result, it suffices to show that \eqref{DBD>1}
and \eqref{DBD<1} hold on $\cO_q^\reg(\CC^n)$.
Suppose first that $|q|>1$. Applying Corollary~\ref{cor:qball_norms_2}
and Lemma~\ref{lemma:binom_est},
for each $f=\sum_k c_k x^k\in\cO_q^\reg(\CC^n)$ we obtain
\[
\begin{split}
\| f\|_{\BB,\rho}
&=\sum_k |c_k| \left(\frac{[k]_{|q|^{-2}}!}{\bigl[ |k|\bigr]_{|q|^{-2}}!}\right)^{1/2}\!\!\!\!
\rho^{|k|}
\le \sum_k |c_k| \rho^{|k|} = \| f\|_{\DD,\rho};\\
\| f\|_{\BB,\rho}
&=\sum_k |c_k| \left(\frac{[k]_{|q|^{-2}}!}{\bigl[ |k|\bigr]_{|q|^{-2}}!}\right)^{1/2}\!\!\!\!
\rho^{|k|}\\
&\ge (|q|^{-2};|q|^{-2})_\infty^{n/2} \sum_k |c_k| \rho^{|k|}
= (|q|^{-2};|q|^{-2})_\infty^{n/2}\| f\|_{\DD,\rho}.
\end{split}
\]
This completes the proof in the case where $|q|>1$. The case $|q|<1$
is reduced to the previous one by applying \eqref{q-qinv_norms}.
\end{proof}

\begin{corollary}
\label{cor:B_inf=C^n}
For each $q\in\CC^\times$, we have $\cO_q(\BB^n_\infty)=\cO_q(\CC^n)$
as vector subspaces of $\CC[[x_1,\ldots ,x]]$ and as Fr\'echet algebras.
\end{corollary}
\begin{proof}
If $|q|\ne 1$, then the result follows from Theorem~\ref{thm:poly_vs_ball_2}.
Suppose now that $|q|=1$.
By looking at \eqref{O_q} and \eqref{poly_power_rep}, we see that
the map from $\cO(\CC^n)$ to $\cO_q(\CC^n)$ that takes each
$f\in\cO(\CC^n)$ to its Taylor series at $0$ is a Fr\'echet space isomorphism.
On the other hand,
Proposition~\ref{prop:ball_pow_ser} yields a similar topological isomorphism
between the underlying Fr\'echet spaces of $\cO(\CC^n)$ and $\cO_q(\BB^n_\infty)$.
By composing these isomorphisms, we obtain a Fr\'echet space isomorphism between
$\cO_q(\CC^n)$ and $\cO_q(\BB^n_\infty)$ taking $x^k$ to $x^k$ ($k\in\Z_+^n$).
Clearly, this is an algebra isomorphism.
\end{proof}

\begin{remark}
\label{rem:q_poly_ball_lcs}
We have already mentioned that, for each $q\in\CC^\times$, the monomials
$x^k\; (k\in\Z_+^n)$ form a basis of $\cO_q^\reg(\CC^n)$. Hence
we have a vector space isomorphism $\cO^\reg(\CC^n)\to\cO_q^\reg(\CC^n)$
given by $x^k\mapsto x^k$. It is natural to ask whether a similar result holds for
$\cO_q(\DD^n_r)$ and $\cO_q(\BB^n_r)$, i.e., whether there exist Fr\'echet space
isomorphisms
\[
\varphi_1\colon \cO(\DD^n_r)\to\cO_q(\DD^n_r), \quad
\varphi_2\colon \cO(\BB^n_r)\to\cO_q(\BB^n_r),
\]
that take each holomorphic function $f$ to its Taylor series at $0$.
If $|q|=1$, then, comparing~\eqref{poly_power_rep} with Definition~\ref{def:q_poly}
and~\eqref{ball_pow_rep} with Definition~\ref{def:qball}, respectively, we see that
both $\varphi_1$ and $\varphi_2$ are Fr\'echet space isomorphisms.
The same argument shows that $\varphi_1$ is a Fr\'echet space isomorphism whenever
$|q|\ge 1$. If $|q|<1$, then $\varphi_1$ is a continuous linear map
(because $w_q(k)\le 1$), but is not an isomorphism by Remark~\ref{rem:no_mult}.
For the same reason, if $|q|<1$, then $\varphi_2$ is a continuous linear map,
but is not an isomorphism. Finally, if $|q|>1$, then $\varphi_2$ is not well defined (unless $r=\infty$).
Indeed, if $\varphi_2$ existed, then we would have a chain of linear maps
\[
\cO(\BB^n_r) \xra{\varphi_2} \cO_q(\BB^n_r)=\cO_q(\DD^n_r) \xra{\varphi_1^{-1}}
\cO(\DD^n_r),
\]
and the composite map $\varphi\colon\cO(\BB^n_r)\to\cO(\DD^n_r)$
would take each $z^k$ to itself. Therefore $\varphi$ would be an inverse for the restriction
map $\cO(\DD^n_r)\to\cO(\BB^n_r)$, which is a contradiction.
\end{remark}

Now let us turn to a more difficult and more interesting case $|q|=1$.
To prove our $q$-version of Poincar\'e's theorem, we will need to extend the notion of
joint spectral radius in Banach algebras (see \cite[V.35 and Comments to Chapter V]{Muller_book})
to the setting of Arens-Michael algebras.

\begin{definition}
Let $A$ be an Arens-Michael algebra, and let $\{ \|\cdot\|_\lambda : \lambda\in\Lambda\}$
be a directed defining family of submultiplicative seminorms on $A$.
Given an $n$-tuple $a=(a_1,\ldots ,a_n)\in A^n$, we define the {\em joint $\ell^p$-spectral radius}
$r^A_p(a)$ by
\begin{equation}
\label{j-sprad}
\begin{split}
r^A_p(a)&=\sup_{\lambda\in\Lambda}%
\lim_{d\to\infty}\Bigl(\sum_{\alpha\in W_{n,d}} \| a_\alpha\|_\lambda^p\Bigr)^{1/pd}
\quad\text{for }1\le p<\infty;\\
r^A_\infty(a)&=\sup_{\lambda\in\Lambda}%
\lim_{d\to\infty}\Bigl(\sup_{\alpha\in W_{n,d}} \| a_\alpha\|_\lambda\Bigr)^{1/d}.
\end{split}
\end{equation}
\end{definition}

\begin{remark}
The joint $\ell^\infty$-spectral radius was studied by A.\,So\l tysiak \cite{Solt2}
in the case where the $a_j$'s commute, but $A$ is not necessarily locally $m$-convex.
\end{remark}

By \cite[C.35.2]{Muller_book}, the limits in \eqref{j-sprad} always exist.
In contrast to the Banach algebra case, it may happen that $r^A_p(a)=+\infty$.
For example, if $A=\cO(\CC)$ and $z\in A$ is the complex coordinate, then an easy
computation shows that $r_p^A(z)=+\infty$ for all $p$
(see also Examples~\ref{ex:j-sprad_qpoly} and~\ref{ex:j-sprad_qball} below).

\begin{prop}
The definition of $r^A_p(a)$ does not depend on the choice of a
directed defining family of submultiplicative seminorms on $A$.
\end{prop}
\begin{proof}
Let $S=\{ \|\cdot\|_\lambda : \lambda\in\Lambda\}$ and
$S'=\{ \|\cdot\|_\mu : \mu\in\Lambda'\}$ be two
directed defining families of submultiplicative seminorms on $A$, and let
$r^A_p(a;S)$ and $r^A_p(a;S')$ denote the respective joint spectral radii.
Then for each $\mu\in\Lambda'$ there exist $\lambda\in\Lambda$ and $C>0$ such that
$\| \cdot\|_\mu \le C\| \cdot\|_\lambda$ on $A$. Therefore for each $p\in [1,+\infty)$ we obtain
\[
\lim_{d\to\infty}\Bigl(\sum_{\alpha\in W_{n,d}} \| a_\alpha\|_\mu^p\Bigr)^{1/pd}
\le \lim_{d\to\infty} C^{1/d}\Bigl(\sum_{\alpha\in W_{n,d}} \| a_\alpha\|_\lambda^p\Bigr)^{1/pd}
\le r_p^A(a;S),
\]
whence $r_p^A(a;S')\le  r_p^A(a;S)$. For $p=\infty$, the computation is similar.
Since $S$ and $S'$ are equivalent, the result follows.
\end{proof}

Given an algebra $A$, an $n$-tuple $a=(a_1,\ldots ,a_n)\in A^n$, and an algebra homomorphism
$\varphi\colon A\to B$, we denote by $\varphi(a)$ the $n$-tuple
$(\varphi(a_1),\ldots ,\varphi(a_n))\in B^n$.

\begin{prop}
\label{prop:j-sprad_hom}
Let $A$ and $B$ be Arens-Michael algebras, and let $\varphi\colon A\to B$ be a continuous
homomorphism. Then for each $a\in A^n$ we have $r_p^B(\varphi(a))\le r_p^A(a)$.
\end{prop}
\begin{proof}
Let $\{ \|\cdot\|_\lambda : \lambda\in\Lambda\}$ and
$\{ \|\cdot\|_\mu : \mu\in\Lambda'\}$ be directed defining families of submultiplicative
seminorms on $A$ and $B$, respectively.
Then for each $\mu\in\Lambda'$ there exist $\lambda\in\Lambda$ and $C>0$ such that
for each $a\in A$ we have $\| \varphi(a)\|_\mu \le C\| a\|_\lambda$.
Let now $a\in A^n$ and $p\in [1,+\infty)$. We have
\[
\begin{split}
\lim_{d\to\infty}\Bigl(\sum_{\alpha\in W_{n,d}} \| \varphi(a)_\alpha\|_\mu^p\Bigr)^{1/pd}
&=\lim_{d\to\infty}\Bigl(\sum_{\alpha\in W_{n,d}} \| \varphi(a_\alpha)\|_\mu^p\Bigr)^{1/pd}\\
&\le \lim_{d\to\infty} C^{1/d}\Bigl(\sum_{\alpha\in W_{n,d}} \| a_\alpha\|_\lambda^p\Bigr)^{1/pd}
\le r_p^A(a).
\end{split}
\]
For $p=\infty$, the computation is similar. The rest is clear.
\end{proof}

\begin{corollary}
\label{cor:j-sprad_isom}
Let $A$ and $B$ be Arens-Michael algebras, and let $\varphi\colon A\to B$ be a topological
algebra isomorphism. Then for each $a\in A^n$ we have $r_p^B(\varphi(a))=r_p^A(a)$.
\end{corollary}

\begin{remark}
If $A$ is a commutative Banach algebra, then for each $a\in A^n$ we have
\begin{equation}
\label{geom_sprad}
r_p^A(a)=\sup\{ \| z\|_p : z\in\sigma_A(a)\},
\end{equation}
where $\sigma_A(a)$ is the joint spectrum of $a$ and $\|\cdot\|_p$ is the $\ell^p$-norm
on $\CC^n$ (see \cite{Solt} or \cite[Theorems 35.5 and 35.6]{Muller_book}). This result easily extends to
commutative Arens-Michael algebras. Indeed, let $A$ be a commutative Arens-Michael algebra, and let
$\{ \|\cdot\|_\lambda : \lambda\in\Lambda\}$ be a directed defining family of submultiplicative
seminorms on $A$. For each $\lambda\in\Lambda$, let $A_\lambda$ denote the completion of $A$
with respect to $\|\cdot\|_\lambda$. By definition, we have
$r^A_p(a)=\sup_\lambda r_p^{A_\lambda}(a_\lambda)$, where $a_\lambda$ is the canonical
image of $a$ in $A_\lambda^n$. On the other hand, a standard argument
(cf. \cite[Proposition 5.1.8]{X2}) shows that
$\sigma_A(a)=\bigcup_\lambda \sigma_{A_\lambda}(a_\lambda)$.
Applying now~\eqref{geom_sprad} to each $a_\lambda$ and taking then the supremum over $\lambda$,
we get the result. In the case where $p=\infty$, a more general fact was proved
by A.\,So\l tysiak \cite{Solt2}.
\end{remark}

Let us introduce some notation. Given $\alpha\in W_n$
and $i\in\{1,\ldots ,n\}$, let
\[
\rp_i(\alpha)=|\alpha^{-1}(i)|.
\]
Thus we have a map
\begin{equation}
\label{proj}
\rp\colon W_n\to\Z_+^n,\quad \rp(\alpha)=(\rp_1(\alpha),\ldots ,\rp_n(\alpha)).
\end{equation}
Observe that, for each $k\in\Z_+^n$, we have
\begin{equation}
\label{card_p^{-1}}
|\rp^{-1}(k)|=\frac{|k|!}{k!}.
\end{equation}
Let now $q\in\CC^\times$, and let $x_1,\ldots ,x_n$ be the canonical generators
of $\cO_q^\reg(\CC^n)$. Then for each $\alpha\in W_n$ there exists
a unique $\rt(\alpha)\in\CC^\times$ such that
\begin{equation}
\label{t_alpha}
x_\alpha=\rt(\alpha)x^{\rp(\alpha)}.
\end{equation}
It is easy to give an explicit formula for $\rt(\alpha)$ (see \cite[Lemma 7.8]{Pir_qball_arx}),
but we do not need it here.
Let us only observe that $\rt(\alpha)$ is an integer power of $q$.

The following two examples will be crucial for what follows.

\begin{example}
\label{ex:j-sprad_qpoly}
Let $|q|=1$, and let $x=(x_1,\ldots ,x_n)\in\cO_q(\DD_r^n)^n$.
We claim that
\begin{equation}
\label{j-sprad_qpoly}
r_2^{\cO_q(\DD_r^n)}(x)=r\sqrt{n}.
\end{equation}
Indeed, for each $\rho\in (0,r)$ we have
\[
\begin{split}
\lim_{d\to\infty} \Bigl(\sum_{\alpha\in W_{n,d}} \| x_\alpha\|^2_{\DD,\rho}\Bigr)^{1/2d}
&=\lim_{d\to\infty} \Bigl(\sum_{\alpha\in W_{n,d}} \| x^{\rp(\alpha)}\|^2_{\DD,\rho}\Bigr)^{1/2d}
=\lim_{d\to\infty} \Bigl(\sum_{\alpha\in W_{n,d}} \rho^{2d}\Bigr)^{1/2d}\\
&=\rho\lim_{d\to\infty} |W_{n,d}|^{1/2d}
=\rho\lim_{d\to\infty} (n^d)^{1/2d}=\rho\sqrt{n}.
\end{split}
\]
Taking the supremum over $\rho$ yields \eqref{j-sprad_qpoly}.
\end{example}

\begin{example}
\label{ex:j-sprad_qball}
Let $|q|=1$, and let $x=(x_1,\ldots ,x_n)\in\cO_q(\BB_r^n)^n$.
We claim that
\begin{equation}
\label{j-sprad_qball}
r_2^{\cO_q(\BB_r^n)}(x)=r.
\end{equation}
To see this, observe that for each $d\in\Z_+$ we have
\begin{equation}
\label{binom_estim}
\bigl| (\Z_+^n)_d\bigr|=\binom{d+n-1}{n-1}
\le (d+1)(d+2)\cdots (d+n-1)\le (d+n-1)^{n-1}.
\end{equation}
Hence for each $\rho\in (0,r)$ we obtain
\[
\begin{split}
\lim_{d\to\infty} \Bigl(\sum_{\alpha\in W_{n,d}} \| x_\alpha\|^2_{\BB,\rho}\Bigr)^{1/2d}
&=\lim_{d\to\infty} \Bigl(\sum_{\alpha\in W_{n,d}} \| x^{\rp(\alpha)}\|^2_{\BB,\rho}\Bigr)^{1/2d}\\
&=\lim_{d\to\infty} \Bigl(\sum_{k\in (\Z_+^n)_d}%
\frac{|k|!}{k!} \| x^k\|^2_{\BB,\rho}\Bigr)^{1/2d}\\
&=\lim_{d\to\infty} \Bigl(\sum_{k\in (\Z_+^n)_d}%
\rho^{2d}\Bigr)^{1/2d}
=\rho\lim_{d\to\infty} |(\Z_+^n)_d|^{1/2d}
=\rho.
\end{split}
\]
Taking the supremum over $\rho$ yields \eqref{j-sprad_qball}.
\end{example}

\begin{remark}
We have already noticed (see Remark~\ref{rem:q_poly_ball_lcs}) that, if $|q|=1$, then
$\cO_q(\DD_r^n)=\cO(\DD^n_r)$ and $\cO_q(\BB_r^n)=\cO(\BB^n_r)$ as locally
convex spaces. Also, observe that $\| x_\alpha\|_{\DD,\rho}$ and
$\| x_\alpha\|_{\BB,\rho}$ do not depend on $q$ provided that $|q|=1$.
Hence $r_2^{\cO_q(\DD_r^n)}(x)$ and $r_2^{\cO_q(\BB_r^n)}(x)$ do not depend on $q$.
On the other hand, it is easy to show that $\sigma_{\cO(\DD_r^n)}(x)=\DD_r^n$
and $\sigma_{\cO(\BB_r^n)}(x)=\BB_r^n$. Applying now~\eqref{geom_sprad}, we obtain
\begin{align*}
r_2^{\cO_q(\DD_r^n)}(x)
&=\sup\{ \| z\|_2 : z\in\DD_r^n\}=r\sqrt{n};\\
r_2^{\cO_q(\BB_r^n)}(x)
&=\sup\{ \| z\|_2 : z\in\BB_r^n\}=r,
\end{align*}
which yields an alternative proof of \eqref{j-sprad_qpoly} and \eqref{j-sprad_qball}.
\end{remark}

Although the algebras $\cO_q(\DD_r^n)$ and $\cO_q(\BB_r^n)$ are not graded in the
purely algebraic sense, it will be convenient to introduce the following terminology.
Let $A$ denote either $\cO_q(\DD_r^n)$ or $\cO_q(\BB_r^n)$, and let
\[
A_i=\spn\{ x^k : |k|=i\} \qquad (i\in\Z_+).
\]
Then each $a\in A$ can be uniquely written as $a=\sum_{i=0}^\infty a_i$,
where $a_i\in A_i$ and the series absolutely converges in $A$. The element $a_i$ will
be called the {\em $i$th homogeneous component} of $a$. Explicitly, if $a=\sum_k c_k x^k$,
then $a_i=\sum_{|k|=i} c_k x^k$. We clearly have $A_i A_j\subseteq A_{i+j}$
for all $i,j\in\Z_+$. Observe also that for each $a,b\in A$ and each $\ell\in\Z_+$ we have
\[
(ab)_\ell=\sum_{i+j=\ell} a_i b_j.
\]
Let also
\[
A_{\ge i}=\ol{\bigoplus_{j\ge i} A_j}=\ol{\spn}\{ x^k : |k|\ge i\} \qquad (i\in\Z_+).
\]
Obviously,
\begin{equation}
\label{ge-incl}
A_{\ge i} A_{\ge j}\subseteq A_{\ge (i+j)} \qquad (i,j\in\Z_+).
\end{equation}

Here is our main result.

\begin{theorem}
If $|q|=1$, $n\ge 2$, and $r<\infty$,
then the Fr\'echet algebras $\cO_q(\DD_r^n)$ and $\cO_q(\BB_r^n)$
are not topologically isomorphic.
\end{theorem}
\begin{proof}
If $q=1$, then the result follows from the classical Poincar\'e theorem
(see the beginning of this section). Thus we may suppose that $q\ne 1$.
Let $A=\cO_q(\DD_r^n)$, $B=\cO_q(\BB_r^n)$, and
assume, towards a contradiction, that $\varphi\colon B\to A$ is a topological
algebra isomorphism. For each $i=1,\ldots ,n$, let $f_i=\varphi(x_i)\in A$
and $g_i=\varphi^{-1}(x_i)\in B$. Given $k\in\Z_+$, let $f_{i,k}$
(respectively, $g_{i,k}$) denote the $k$th homogeneous
component of $f_i$ (respectively, $g_i$).
We claim that
\begin{equation}
\label{f0=0}
f_{i,0}=g_{i,0}=0 \qquad (i=1,\ldots, n).
\end{equation}
Indeed, assume that $f_{i,0}\ne 0$ for some $i$, and fix any $j\ne i$.
Then we have
\begin{equation}
\label{f_i_f_j}
f_i f_j=q' f_j f_i \qquad (\text{where $q'=q$ or $q'=q^{-1}$}).
\end{equation}
Let now $k=\min\{\ell\in\Z_+ : f_{j,\ell}\ne 0\}$. Taking the $k$th homogeneous
components of~\eqref{f_i_f_j}, we obtain $f_{i,0} f_{j,k}=q' f_{j,k} f_{i,0}$,
whence $f_{j,k}=0$. The resulting contradiction implies that $f_{i,0}=0$ for all $i$.
A similar argument shows that $g_{i,0}=0$ for all $i$.

Thus for each $i=1,\ldots ,n$ we have $\varphi(x_i)\in A_{\ge 1}$. Using~\eqref{ge-incl}, we
see that for each $k\in\Z_+^n$
\[
\varphi(x^k)=\varphi(x_1)^{k_1}\cdots\varphi(x_n)^{k_n}\in A_{\ge 1}^{|k|}\subseteq A_{\ge |k|}.
\]
Similarly, $\varphi^{-1}(x^k)\in B_{\ge |k|}$. Hence for each $m\in\Z_+$ we have
\begin{equation}
\label{varphi_ge_incl}
\varphi(B_{\ge m})\subseteq A_{\ge m},\quad
\varphi^{-1}(A_{\ge m})\subseteq B_{\ge m}.
\end{equation}
Let now $\MM_n$ denote the algebra of all complex $n\times n$-matrices, and
let $\alpha=(\alpha_{ij})\in\MM_n$ and $\beta=(\beta_{ij})\in\MM_n$ be such that
\begin{equation}
\label{alpha_beta}
f_{i,1}=\sum_j \alpha_{ij} x_j,\quad g_{i,1}=\sum_j \beta_{ij} x_j.
\end{equation}
Using \eqref{f0=0} and \eqref{varphi_ge_incl}, we obtain
\[
\begin{split}
x_i
&=\varphi(g_i)
\in\varphi(g_{i,1}+B_{\ge 2})
\subseteq\varphi(g_{i,1})+A_{\ge 2}
=\sum_j \beta_{ij} f_j+A_{\ge 2}\\
&=\sum_j \beta_{ij} f_{j,1}+A_{\ge 2}
=\sum_{j,k} \beta_{ij} \alpha_{jk} x_k+A_{\ge 2}
=\sum_k (\beta\alpha)_{ik} x_k+A_{\ge 2}.
\end{split}
\]
Hence $\beta\alpha=1$ in $\MM_n$. In particular, $\alpha$ is invertible.

Fix now $i,j\in\{ 1,\ldots ,n\}$ with $i<j$. Since $f_i f_j=qf_j f_i$ and
$f_{i,0}=f_{j,0}=0$, it follows that $f_{i,1} f_{j,1}=q f_{j,1} f_{i,1}$.
Equivalently,
\[
\sum_k \alpha_{ik} x_k \sum_{\ell} \alpha_{j\ell} x_\ell
=q\sum_k \alpha_{jk} x_k \sum_{\ell} \alpha_{i\ell} x_\ell.
\]
Comparing the coefficients at $x_m^2$ yields $\alpha_{im}\alpha_{jm}=q\alpha_{jm}\alpha_{im}$,
whence $\alpha_{im}\alpha_{jm}=0$ for all $m=1,\ldots ,n$ and for all $i<j$.
In other words, each column of $\alpha$ contains at most one nonzero element.
Since $\alpha$ is invertible, we conclude that there exists a permutation $\sigma$
of $\{ 1,\ldots ,n\}$ such that $\alpha_{ij}=0$ whenever $i\ne\sigma(j)$,
and $\alpha_{\sigma(j)j}\ne 0$. Let $\tau=\sigma^{-1}$, and let $\lambda_i=\alpha_{i\tau(i)}$.
Since $\beta=\alpha^{-1}$, it follows that
\begin{align*}
\alpha_{i\tau(i)}&=\lambda_i\ne 0,\quad \alpha_{ij}=0\quad (j\ne\tau(i));\\
\beta_{i\sigma(i)}&=\lambda_{\sigma(i)}^{-1}\ne 0,\quad \beta_{ij}=0\quad (j\ne\sigma(i)).
\end{align*}
Together with~\eqref{f0=0} and~\eqref{alpha_beta}, this implies that
\begin{equation}
\label{f_i_g_i_incl}
\varphi(x_i)=f_i\in\lambda_i x_{\tau(i)}+A_{\ge 2},
\quad \varphi^{-1}(x_i)=g_i\in\lambda_{\sigma(i)}^{-1} x_{\sigma(i)} +B_{\ge 2}.
\end{equation}
Therefore for each $d\in\Z_+$ we have
\[
\varphi(x_i^d)\in\lambda_i^d x_{\tau(i)}^d + A_{\ge(d+1)},
\quad \varphi^{-1}(x_i^d)\in\lambda_{\sigma(i)}^{-d} x_{\sigma(i)}^d +B_{\ge(d+1)},
\]
whence for each $\rho\in (0,r)$ we obtain
\begin{align}
\label{x_i^d_1}
\| \varphi(x_i^d) \|_{\DD,\rho}
&\ge \| \lambda_i^d x_{\tau(i)}^d \|_{\DD,\rho}
=|\lambda_i|^d \rho^d,\\
\label{x_i^d_2}
\| \varphi^{-1}(x_i^d) \|_{\BB,\rho}
&\ge \| \lambda_{\sigma(i)}^{-d} x_{\sigma(i)}^d \|_{\BB,\rho}
=|\lambda_{\sigma(i)}^{-1}|^d \rho^d.
\end{align}
Fix now $\rho\in (0,r)$, and choose $\omega(\rho)\in (\rho,r)$
and $C(\rho)>0$ such that
\[
\| \varphi(b)\|_{\DD,\rho} \le C(\rho) \| b\|_{\BB,\omega(\rho)} \qquad (b\in B).
\]
Letting $b=x_i^d$ and using~\eqref{x_i^d_1}, we see that
\begin{equation}
\label{lambda_i_est}
|\lambda_i|^d \rho^d \le \| \varphi(x_i^d) \|_{\DD,\rho}
\le C(\rho) \| x_i^d\|_{\BB,\omega(\rho)}=C(\rho)\omega(\rho)^d.
\end{equation}
Raising \eqref{lambda_i_est} to the power $1/d$ and letting then $d\to\infty$,
we conclude that $|\lambda_i|\le\omega(\rho)/\rho$. Letting now $\rho\to r$,
we obtain $|\lambda_i|\le 1$. Applying the same argument to $\varphi^{-1}$
and using~\eqref{x_i^d_2} instead of~\eqref{x_i^d_1}, we see that
$|\lambda_{\sigma(i)}^{-1}|\le 1$. Finally, $|\lambda_i|=1$ for all $i=1,\ldots ,n$.

Given $\alpha=(\alpha_1,\ldots ,\alpha_d)\in W_n$, let
$\tau(\alpha)=(\tau(\alpha_1),\ldots ,\tau(\alpha_d))\in W_n$.
Using again \eqref{f_i_g_i_incl}, we see that
\[
f_\alpha\in \lambda_\alpha x_{\tau(\alpha)} + A_{\ge (|\alpha|+1)} \qquad (\alpha\in W_n),
\]
whence
\[
\| f_\alpha\|_{\DD,\rho}\ge \| \lambda_\alpha x_{\tau(\alpha)}\|_{\DD,\rho}
=\| x_\alpha\|_{\DD,\rho} \qquad (\alpha\in W_n,\;\rho\in (0,r)).
\]
Therefore $r_2^A(f_1,\ldots ,f_n)\ge r_2^A(x_1,\ldots ,x_n)$.
Combining this with~\eqref{j-sprad_qpoly}, \eqref{j-sprad_qball}, and using
Corollary~\ref{cor:j-sprad_isom}, we see that
\[
r=r_2^B(x_1,\ldots ,x_n)=r_2^A(f_1,\ldots ,f_n)
\ge r_2^A(x_1,\ldots ,x_n)=r\sqrt{n}.
\]
The resulting contradiction completes the proof.
\end{proof}

We conclude the paper with an open problem
related to the notion of an HFG algebra \cite{Pir_ncStein,Pir_HFG}.
Let $\cF_n$ denote the Arens-Michael envelope of the free algebra on $n$ generators.
A Fr\'echet algebra $A$ is said to be {\em holomorphically finitely generated} (HFG for short)
if $A$ is isomorphic to a quotient of $\cF_n$ for some $n$.
There is also an ``internal'' definition given in terms of J.\,L.\,Taylor's free functional
calculus. By \cite[Theorem 3.22]{Pir_HFG},
a commutative Fr\'echet-Arens-Michael algebra is holomorphically finitely generated if and only if
it is topologically isomorphic to $\cO(X)$ for some Stein space $(X,\cO_X)$ of finite
embedding dimension. Together with O.\,Forster's theorem \cite{For},
this implies that the category of commutative HFG algebras is anti-equivalent to the
category of Stein spaces of finite embedding dimension. There are many natural examples
of noncommutative HFG algebras \cite[Section 7]{Pir_HFG}.
For instance, $\cO_q(\DD^n_r)$ is an HFG algebra.
By Theorem~\ref{thm:poly_vs_ball_2}, $\cO_q(\BB^n_r)$ is an HFG algebra
provided that $|q|\ne 1$.

\begin{problem}
Is $\cO_q(\BB^n_r)$ an HFG algebra in the case where $|q|=1,\; q\ne 1$?
\end{problem}

\end{document}